\documentclass[a4paper]{amsart}
\usepackage[utf8]{inputenc}
\usepackage[a4paper,margin=2.5cm]{geometry}
\usepackage{amsmath,amssymb,amsthm,mathtools,abraces,stmaryrd}
\usepackage{color,xcolor}
\usepackage[hypertexnames=false,hyperfootnotes=false,colorlinks=true,linkcolor=blue,%
citecolor=purple,filecolor=magenta,urlcolor=cyan,unicode,linktocpage=true]{hyperref}
\usepackage[backend=bibtex, style=numeric, maxbibnames=6]{biblatex}
\renewbibmacro{in:}{}
\newtheorem{theorem}{Theorem}  
\newtheorem*{theorem*}{Theorem} 
 
\newtheorem*{proposition*}{Proposition} 
\newtheorem*{lemma*}{Lemma}

\theoremstyle{definition}

\theoremstyle{remark}

\newtheorem*{remark*}{Remark}

\newtheoremstyle{cited}%
  {6pt}
  {6pt}
  {\itshape}
  {}
  {\bfseries}
  {\textbf{.}}
  {.5em}
  {\thmname{#1}  \thmnote{\normalfont#3}}

\theoremstyle{cited}
\newtheorem*{citedthm}{Theorem}

\newtheoremstyle{citedremark}%
  {5pt}
  {5pt}
  {\itshape}
  {}
  {\it}
  {.}
  {.5em}
  {\thmname{#1} \thmnote{\normalfont#3}}
\theoremstyle{citedremark}
\newtheorem*{citedrmk}{Remark}

\newcommand{\Q}{\mb{Q}}

\def\={\;=\;}  \def\+{\,+\,}
\def\be{\begin{equation}}  \def\ee{\end{equation}}
\def\coloneqqq{\;\coloneqq\;}

\def\={\;=\;} \def\+{\,+\,}     \def\Q{\Bbb Q}   
  \def\W{\widetilde W}    \def\p{\partial}
\def\Res#1{\underset{#1}{\mathrm{Res}}}

\title{A curious identity that implies Faber's conjecture}
\author{Elba Garcia-Failde}
\address{\parbox{\linewidth}{Institut des Hautes \'{E}tudes Scientifiques, 35 Route de Chartres, 91440 Bures-sur-Yvette, France\\
and Universit\'{e} de Paris, B\^{a}timent Sophie Germain, 8 Place Aur\'{e}lie Nemours, Paris, France \vspace{0.15cm}}}
\email{garciafailde@ihes.fr}
\author{Don Zagier}
\address{\parbox{\linewidth}{Max-Planck-Institut f\"{u}r Mathematik, Vivatsgasse 7, Bonn 53111, Germany\\
and International Centre for Theoretical Physics, Strada Costiera 11, Trieste 34014, Italy \vspace{0.15cm}}}
\email{dbz@mpim-bonn.mpg.de}

\begin{document}

\begin{abstract}
We prove that a curious generating series identity implies Faber's intersection number conjecture
(by showing that it implies a combinatorial identity already given in~\cite{GFKLS19}) and give a new proof of 
Faber's conjecture by directly proving this identity.



\end{abstract}

\maketitle

We recall one of the equivalent forms of Faber's conjecture, now a theorem, on proportionalities 
of kappa-classes on the moduli space $\mathcal{M}_g$ of curves of genus~$g\geq 2$:

\begin{theorem*}[Faber's Intersection Number Conjecture~\cite{Faber1999}]\label{thm:Faber}
	Let $n \geq 2$ and $g \geq 2$. For any $d_1,\dots, d_n\geq 1$, $d_1+\cdots+d_n = g-2+n$, there exists a constant $C_g$ that only depends on $g$ such that
	\begin{equation}\label{Faber_ctt}
	\frac {1}{(2g-3+n)!}
	\int_{\overline{\mathcal{M}}_{g,n}} \lambda_g\lambda_{g-1} \prod_{i=1}^n \psi_i^{d_i} (2d_i-1)!! \= C_g\,.
	\end{equation}
\end{theorem*}

\begin{citedrmk}[\cite{Faber1999}] 
\emph{From the known value of $\int_{\overline{\mathcal{M}}_{g,1}} \lambda_g\lambda_{g-1}\psi_1^{g-1}$
one deduces the value $C_g = \frac{|B_{2g}|}{2^{2g-1}(2g)!}$, where $B_{2g}$ is the $(2g)$th Bernoulli number.}
\end{citedrmk}

This theorem has now been proved in several different ways. Getzler and Pandharipande~\cite{GetPan} derived it from the 
Virasoro constrains for~$\mathbb{P}^2$, later proved by Givental~\cite{Givental2001MMJ}. Liu and Xu~\cite{LiuXu} derived it 
from an identity for the $n$-point functions of the intersection numbers of $\psi$-classes that comes from the KdV equation. 
Goulden, Jackson, and Vakil proved it for $n\leq 3$ using degeneration and localization of Faber--Hurwitz classes~\cite{GouJackVak}. 
Buryak and Shadrin proved it using relations for double ramification cycles~\cite{BuryakShadrin}. Finally, 
Pixton showed the compatibility of this theorem with Faber--Zagier relations in~\cite{PixtonThesis}, also proved by 
Faber and the second author (unpublished, see a remark in~\cite{PPZ16}). Together with a result of~\cite{PPZ16}, 
this shows that the Faber--Zagier relations imply this theorem. 
The proof we will give relies instead on the following equivalence from~\cite{GFKLS19}, which was found via the so-called 
half-spin tautological relations~\cite{KLLS17}:
\begin{citedthm}[\cite{GFKLS19}] 
Faber's intersection number conjecture is equivalent to the following system of combinatorial identities:
For any $g, n\geq 2$ and $a_1,\dots,a_n\in\mathbb{Z}_{\geq 0}$ with $a_1+\cdots+a_n=2g-3+n$, 
\begin{align} \label{eq:keyeq}
	& 0 \= \sum_{k=1}^n \frac{(-1)^k(2g-3+k)!}{k!}
	\sum_{\substack{I_1\sqcup\cdots\sqcup I_k = \llbracket n \rrbracket\\ I_j\neq \emptyset, \forall j\in \llbracket k \rrbracket}}
	\sum_{\substack{d_1,\dots,d_k \in\mathbb{Z}_{\geq 0} \\ d_1+\cdots+d_k=g-2+n}}
	\prod_{j=1}^k \binom{2a_{[I_j]}+1}{2d_j} \frac{(2d_j-1)!!}{(2d_j+1-2|I_j|)!!}\,.
\end{align}
\noindent Here by $a_{[I_j]}$ we denote $\sum_{\ell\in I_j} a_\ell$ and by $|I_j|$ we denote the cardinality of the 
set $I_j\subset \llbracket n \rrbracket$, $j=1,\dots,k$.
\end{citedthm}

Since this theorem is an equivalence and Faber's conjecture is proved, we know that the combinatorial 
identity~\eqref{eq:keyeq}, which was already verified for $2\le n\le5$ in~\cite{GFKLS19}, must hold for all~$n\ge2$. 
Our goal  is to give an independent elementary proof of it. Specifically, we will show in  Section~\ref{S1} 
that~\eqref{eq:keyeq} is a consequence of the following curious identity, whose proof will then be given in Section~\ref{S2}.

\begin{theorem}\label{mainThm}
Let $A(v,y) = v^{-1}\,P(v(1+y)^2)$, where $P(X)=\sum_{a\ge0}c_aX^a$ is a polynomial with
infinitesimal coefficients.  Define a power series $\,T(v,y)$ by
  $$ T(v,y) \coloneqqq y \+
   \sum_{r=1}^\infty\frac1{r!}\,\Bigl(\frac1y\,\frac d{dy}\Bigr)^{r-1}\Bigl(\frac{1+y}y\,A(v,y)^r\Bigr)\,.$$
Then for every positive even integer~$N$ we have
\begin{equation}\label{mainId}
\bigl[v^{N-1}y^{-2}\bigr]\biggl(\frac{T(v,y)-T(v,-y)}2\biggr)^{-N} \= -\,\frac{(2N+1)!}{(N-1)!(N+1)!}\,c_N. 
\end{equation}
\end{theorem}  
We make several remarks about this statement.  First of all, when we say that the $c_a$'s are infinitesimal, we simply mean all expressions being considered, such as $T(v,y)$, belong to the formal power series ring $\Q[v^{\pm1},y^{\pm1}][[c_0,c_1,\dots]]$ and there are therefore no convergence issues.
Secondly, the restriction to $N$ positive
and even is harmless since the left-hand side of~\eqref{mainId} vanishes trivially for $N=0$ or $N$ odd. Finally,
equation~\eqref{mainId} is slightly stronger than what we need to prove~\eqref{eq:keyeq}, for which it would
suffice to know that the coefficient on the left is linear in the coefficients~$c_a$, i.e., that it has no
terms of total degree $\geq 2$ in the $c_a$'s. 

\section{Reduction to a curious identity}\label{S1}

Theorem~\ref{mainThm} will be proved in \S\ref{S2}. In this section we will show how it implies the identity~\eqref{eq:keyeq}
(and hence Faber's conjecture).  To do this, we will rewrite the right-hand side of~\eqref{eq:keyeq} in 
terms of a simpler expression defined using generating functions. 

\begin{proposition*}\label{reduction} 
Let $P(X) = P(x_1,\dots,x_n;X)=\sum_{\ell=1}^n x_\ell\,X^{a_{\ell}}$ and define $S(v,y)=S(x_1,\dots,x_n;v,y)$ by
\begin{equation}\label{defS}
S(v,y) \= y\+\sum_{r=1}^n \frac{1}{r!} \Bigl(\frac{1}{y}\frac{\mathrm{d}}{\mathrm{d} y}\Bigr)^{r-1}
       \biggl(\frac{(1+y)P(v(1+y)^2)^r+(1-y)P(v(1-y)^2)^r}{2y}\biggr)\,.
\end{equation}
Then the right-hand side of~\eqref{eq:keyeq} is equal to $(2g-3)!\,[x_1\cdots x_n\,v^{2g+n-3}\,y^{-2}]\bigl(S(v,y)^{2-2g}\bigr)$.
\end{proposition*}
\begin{proof} 
We first note that dividing by $(2g-3)!$ replaces the prefactor $\frac{(-1)^k(2g-3+k)!}{k!}$ in~\eqref{eq:keyeq} by 
the simpler binomial coefficient $\binom{2-2g}k$, and also that the sum from $k=1$ to~$n$ can be replaced by a sum over 
all~$k\ge0$ since a decomposition $I_1\sqcup\cdots\sqcup I_k = \llbracket n \rrbracket$ with all $I_j$ non-empty can 
only exist if $1\le k\le n$.  Then we introduce a formal variable~$y$ and use the equality
$$
\frac{(2d_j-1)!!}{(2d_j+1-2|I_j|)!!}\,y^{2d_j+1-2|I_j|} 
  \= \Big(\frac{1}{y}\frac{\mathrm{d}}{\mathrm{d} y}\Big)^{|I_j|-1}y^{2d_j-1} 
$$
to write the inner sum in~\eqref{eq:keyeq} for given $a_\ell$ and $I_j$ as the coefficient of $y^{2g-4+k}$ in
$$ \prod_{j=1}^k \biggl[\Big(\frac1y\frac{\mathrm{d}}{\mathrm{d} y}\Big)^{|I_j|-1}
\sum_{d=0}^\infty \binom{2a_{[I_j]}+1}{2d} y^{2d-1}\biggr]
\=\prod_{j=1}^k \biggl[\Big(\frac1y\frac{\mathrm{d}}{\mathrm{d} y}\Big)^{|I_j|-1}
\frac{(1+y)^{2a_{[I_j]}+1}+(1-y)^{2a_{[I_j]}+1}}{2y}\biggr]\,.
$$
(Extracting the coefficient of $y^{2g-4+k}$ corresponds to the condition $d_1+\cdots+d_k=g-2+n$.)  We now
introduce $n+1$ further formal variables $x_1,\dots,x_n$ and~$v$ and use the identity
$$ \sum_{\substack{I_1\sqcup\cdots\sqcup I_k=\llbracket n\rrbracket \\ I_1,\dots,I_k\neq\emptyset}}F(I_1)\cdots F(I_k)
\= [x_1\cdots x_n]\Biggl(\,\sum_{\emptyset\ne I\subseteq\llbracket n\rrbracket}F(I)\,x_{\{I\}}\Biggr)^k\,, $$
where $x_{\{I\}}$ stands for $\prod_{\ell\in I} x_{\ell}$, which is valid for any function $F$ on the power set of $\llbracket n\rrbracket$, to write the quotient 
 of~\eqref{eq:keyeq} by $(2g-3)!$ as the coefficient of $x_1\cdots x_n\,y^{2g-4}\,v^{2g+n-3}$ in
\begin{align*} &\sum_{k=0}^\infty\binom{2-2g}k\,\biggl(\frac1y\sum_{\emptyset \neq I\subset \llbracket n \rrbracket} 
  x_{\{I\}}\,v^{a_{[I]}}\,\Big(\frac{1}{y}\frac{\mathrm{d}}{\mathrm{d} y}\Big)^{|I|-1}
       \Bigl(\frac{(1+y)^{2a_{[I]}+1}+(1-y)^{2a_{[I]}+1}}{2y}\Bigr)\biggr)^k \\
  &\qquad \= \biggl(1 \+ \frac1y\,\sum_{\emptyset \neq I\subset \llbracket n \rrbracket} 
  x_{\{I\}}\,v^{a_{[I]}}\,\Big(\frac{1}{y}\frac{\mathrm{d}}{\mathrm{d} y}\Big)^{|I|-1}
       \biggl(\frac{(1+y)^{2a_{[I]}+1}+(1-y)^{2a_{[I]}+1}}{2y}\biggr)\biggr)^{2-2g}\,.
\end{align*}  
(Here extracting the coefficient of $v^{2g-3+n}$ corresponds to the condition $a_1+\cdots+a_n=2g-3+n$.)~The 
proposition then follows by noting that the coefficient of $x_1\cdots x_n$ in a polynomial or power series
depends only on its congruence class of modulo the ideal generated by $x_\ell^{\,2}$ ($\ell=1,\dots,n$) and that,
if we denote this equivalence relation by~$\equiv$, we have
$$ \sum_{\substack{I\subseteq\llbracket n\rrbracket \\ |I|=r}} x_{\{I\}}\,v^{a_{[I]}}\,(1\pm y)^{2a_{[I]}}
\;\equiv\;\frac1{r!}\,P(v(1\pm y)^2)^r\,, $$
for each $0\le r\le n$, since each term $\prod_{\ell\in I}x_\ell\,v^{a_\ell}(1\pm y)^{2a_\ell}$ appears 
$r!$~times in $P\bigl(v(1\pm y)^2\bigr)^r$. The last identity can also be justified by observing that the LHS is 
the coefficient of $t^r$ in $\prod_{\ell=1}^n\bigl(1+tx_\ell v^{a_\ell}(1\pm y)^{2a_\ell}\bigr)$, which is congruent to
$\exp\bigl(tP(v(1\pm y)^2)\bigr)$.
\end{proof}

\smallskip
The identity~\eqref{eq:keyeq} for $g,\,n\ge2$ follows immediately by combining Theorem~\ref{mainThm} (with $N=2g-2$) and the proposition (with
the same~$P$, so with $c_a$ equal to $\sum_{a_\ell=a}x_\ell$), since if we
rescale $x_1,\,\dots,\,x_n$ in~\eqref{defS} by dividing them by~$v$ then the expression whose vanishing we have 
to prove is just the coefficient of $x_1\cdots x_n$ in the left-hand side of~\eqref{mainId}, which vanishes for
$n>1$ because $c_N$ is linear in the~$x$'s. 

\section{Proof of the curious identity}\label{S2}

In this section we prove Theorem~\ref{mainThm}.  The first step is to give a different expression for the power
series $T(v,y)$ appearing there.  This is done in the following lemma, in which there is no parameter~$v$.
\begin{lemma*}\label{Lagrange}
Let $A(y)$ be a polynomial with infinitesimal coefficients.  Then 
\begin{equation*} \label{Tdef}
 y\+\sum_{r\geq 1} \frac1{r!}\Big(\frac1y\frac{\mathrm{d}}{\mathrm{d} y}\Big)^{r-1}\Big(\frac{1+y}{y}A(y)^r\Big) 
   \=w\+A(w)\,,
\end{equation*}
where $w$ is the solution near y=w of $w^2=y^2+2A(w)$.
\end{lemma*}
\begin{proof} This is in principle just an application of Lagrange's inversion theorem, but we give a proof via a residue 
calculation. 
As with Theorem~\ref{mainThm}, ``polynomial with infinitesimal coefficients" means that all expressions being considered are to be interpreted  as formal power series (with coefficients in the ring of Laurent polynomials in $y$) in the coefficients of $A$. The easiest way to keep track of everything is to replace $A(y)$ by $xF(y)$, where 
$F$ is a polynomial, so that the powers of $x$ keep track of the degree of the terms with respect to the coefficients of $A$.
Then setting $T=w+xF(w)=w+\frac{w^2-y^2}2$  and using residue calculus (with
$y$~fixed and $x$~variable), we find:
$$ [x^r](T) \= \Res{x=0}\Bigl(\frac{w+w^2/2}{x^{r+1}}\,\mathrm{d}x\Bigr)  
\= \frac1r\,\Res{w=y}\Bigl(\frac{1+w}w\,F(w)^r\,\frac{\mathrm{d} z}{z^r}\Bigr)
\= \frac1{r!}\Big(\frac1y\frac{\mathrm{d}}{\mathrm{d} y}\Big)^{r-1}\Big(\frac{1+y}{y}F(y)^r\Big)
$$
for $r>0$, where we have used the local parameter $z=xF(w)=\frac{w^2-y^2}2$,  $\frac {\mathrm{d}}{\mathrm{d}z}=\frac1y\,\frac {\mathrm{d}}{\mathrm{d}y}$, near~$w=y$.
\end{proof}
 
\begin{proof}[Proof of Theorem~\ref{mainThm}]
Applying the lemma to $A(v,y)$, we find that $T(v,\pm y)=w_\pm +A(v,w_\pm)$, where $w_\pm$ is 
the solution of $w^2-2A(v,w)=y^2$ near to~$\pm y$. Our goal is to show that $[v^{-1}y^{-2}]S^{-N}=-\frac{(2N+1)!}{(N-1)!(N+1)!}c_N$
for $N>0$ even, where $S\coloneqq v\frac{T(v,y)-T(v,-y)}2$ and $A(v,y)=\sum_a c_av^{a-1}(1+y)^{2a}$.  The first
step is to note that, again by residue calculus, for fixed~$v$ we have
$$
[y^{-2}]\,S^{-N} =\Res{y=0}\,\, \frac{y\, \mathrm{d}y}{S^N} = \Res{y=0}\,\, \frac{\mathrm d(y^2/2)}{S^N}
  =-\frac12 \,\Res{S=0}\,\, y^2\, \mathrm{d}(S^{-N}) =\frac N2\,\Res{S=0}\,\, \frac{y^2\,\mathrm d S}{S^{N+1}}=\frac{N}{2}[S^N]\,y^2\,.
$$
Hence the identity to be proved can also be written as $[v^0S^N](Y) = -\binom{2N+2}{N+1}c_N$ for $N>0$ even, where~$Y=(y^2-1)v$.
(As already mentioned in the introduction, this coefficient is trivially~0 if $N$ is zero or odd.)
We define new variables $V,\,W,\,\W$ by $V=\sqrt v$, $W=(1+w_+)V$, $\W=(1+w_-)V$.  
Then $S$,~$V$ and~$Y$ all become polynomials in $W$ and~$\W$, namely 
  $$ S\=\frac{W^2-\W^2}4\,, \qquad V=\frac12 \frac{Q(W)-Q(\W)}{W-\W},
  \qquad Y \= \frac{WQ(\W)-\W Q(W)}{W-\W}\,, $$
where $Q(X)=X^2-2P(X^2)$. Now change variables again by $(W,\W)=(r+s,r-s)$ and set $Q^\pm=Q(r\pm s)$. 
Then a simple computation shows that
  $$ S\=rs, \;\quad V \= \frac{Q^+-Q^-}{4s}\,, \;\quad Y \= \frac{Q^++Q^-}{2}\,-\,2rV $$ 
and the quantity that we want to compute is 
  $$ [V^0S^N](Y) \= \Res{V=0}\;\Res{S=0}\,\Bigl(Y\,\frac{dV}V\,\frac{dS}{S^{N+1}}\Bigr) \= \bigl[r^Ns^N\bigr]\Bigl(\frac{JY}V\Bigr)\,, $$
where $J = rV_r-sV_s$, with $V_r=\p V/\p r$ and $V_s=\p V/\p s$, is the Jacobian of the transformation $(r,s)\mapsto(V,S)$.  We have
$$ \frac{JY}{V} \= 2r\,\bigl(sV_s-rV_r\bigr) \+ \frac{rV_r-sV_s}{2V}\,\bigl(Q^++Q^-\bigr) \= 2r\,\bigl(sV_s-rV_r\bigr) \+\Bigl(\frac12 \+\frac{r(sV)_r-s(sV)_s}{2sV}\Bigr)(Q^++Q^-)\,.  
$$ 
The coefficient of $r^Ns^N$ in the first two terms is easily computed in closed form and is given by
  $$\bigl[r^Ns^N\bigr]\biggl(2r\,\bigl(sV_s-rV_r\bigr) + \frac{Q^++Q^-}{2}\biggr) \= -2c_N\biggl(\binom{2N}{N-1}+\binom{2N}{N}\biggr) \=-\,\binom{2N+2}{N+1}c_N\,,  $$  
and the two numbers $[r^Ns^N]\bigl(\frac{r(sV)_r-s(sV)_s}{sV}\,Q^\pm\bigr)$ both vanish because they are diagonal
coefficients of power series in $r$ and~$s$ that are antisymmetric under interchange of the two variables.
(To see that $(r(sV)_r-s(sV)_s)/(sV)$ is a power series in $r$ and~$s$, note that $Q$ is an even polynomial with non-vanishing 
quadratic term, so $sV$ is $rs$ times a polynomial with non-vanishing constant term and  $r(sV)_r-s(sV)_s$ 
is divisible by~$rs$.)
\end{proof}

\section{Another curious identity}


In the course of finding and proving Theorem~\ref{mainThm} we empirically discovered the following result, 
which seems interesting enough to be worth stating, even though we don't know of any applications, since it may 
indicate that there are much more general identities of this sort. 
\begin{theorem}\label{conj2} Let all notations be as in Theorem~\ref{mainThm}. Then for all $N\geq 1$ one has
\begin{equation}\label{eq2}
\bigl[v^{N-1}y^{-2}\bigr] \biggl(\frac{T(v,y)+y}{2}\biggr)^{-N}\= -\,\frac{N}{4}\,\binom{2N+2}{N+1} c_{N}\,.
\end{equation}
\end{theorem}

\begin{proof}[Proof of Theorem~\ref{conj2}] The proof follows the same lines as that of Theorem~\ref{mainThm}.
We again set $V=\sqrt{v}$, $W=(1+w)V,\;Y=(y^2-1)v$ and want to evaluate $R_N\coloneqq[V^0 S^N](Y)$, but now 
with $S \coloneqq v(T(v,y)+y)/2$. This time we define the new local coordinates $r$ and~$s$ by $r=(W+(1-y)V)/2$ 
and $s=(W-(1-y)V)/2$, so that the variables $(r,s)$ are again related to $W$ and~$S$ by $r+s = W$ and $rs = S$.
The expressions $sV=rs-P((r+s)^2)/2$ and $Y=(r-s)^2-2(r-s)V$ are now different polynomials in $r$ and $s$, but we still have 
$R_N = [r^N s^N](JY/V)$ and $J = rV_r-sV_s$. 
Then $JY/V=-2(r-s)J+(r-s)^2(1+(r(sV)_r-s(sV)_s)/(sV))$. The first two terms
are easy and give the desired binomial coefficient times $c_N$, and the last one is antisymmetric and 
hence gives~$0$.
The only slightly tricky point is that $(r(sV)_r - s(sV)_s)/(sV)$  is no longer a power 
series in $r$ and~$s$, but instead a power series in the infinitesimal variables whose
coefficients are Laurent polynomials in $r$ an~$s$, rather than polynomials as before.
\end{proof}

\medskip
\section*{Acknowledgments}
The first author was supported by the public grant ``Jacques Hadamard'' as part of the 
Investissement d'avenir project, reference ANR-11-LABX-0056-LMH, LabEx LMH and currently receives 
funding from  the  European  Research Council  (ERC)  under  the  European  Union's Horizon  2020  
research and  innovation  programme  (grant  agreement  No.~ERC-2016-STG 716083  ``CombiTop''). 
She is also grateful to the Max Planck Institute for Mathematics in Bonn and the International 
Centre for Theoretical Physics for visits that made this work possible, and to R.~Kramer, 
D.~Lewa\'nski and S.~Shadrin for useful discussions.

\printbibliography


\end{document}